\documentclass[12pt]{article}

\usepackage{fullpage,amssymb,amsthm,amsfonts,enumerate,textcomp, eurosym, titling, epsfig,epstopdf, authblk,pinlabel,xcolor,caption, graphicx, tikz-cd, xcolor}
\usepackage{subcaption}

\usepackage{amsmath}
\makeatletter
\def\grd@save@target#1{%
  \def\grd@target{#1}} \def\grd@save@start#1{%
  \def\grd@start{#1}}

\begin{document}

\makeatother

\title{Turaev--Viro invariants and cabling operations}
\author{Sanjay Kumar and Joseph M. Melby}
\date{\vspace{-5ex}}

\newcommand{\Addresses}{{
  \bigskip
  \footnotesize

  \textsc{Department of Mathematics, University of California, Santa Barbara, Santa Barbara,  CA,
  93106-6105, USA}\par\nopagebreak \textit{E-mail address}: \texttt{sanjay\_kumar@ucsb.edu}\\
 
  \textsc{Department of Mathematics, Michigan State University, East Lansing, MI, 48824,
    USA}\par\nopagebreak \textit{E-mail address}: \texttt{melbyjos@msu.edu}

}}

\maketitle


\begin{abstract}
In this paper, we study the variation of the Turaev--Viro invariants for $3$-manifolds with toroidal
boundary under the operation of attaching a $(p,q)$-cable space. We apply our results to a
conjecture of Chen and Yang which relates the asymptotics of the Turaev--Viro invariants to the
simplicial volume of a compact oriented $3$-manifold. For $p$ and $q$ coprime, we show that the
Chen--Yang volume conjecture is stable under $\left(p,q\right)$-cabling. We achieve our results
by studying the linear operator $RT_r$ associated to the torus knot cable spaces  by the
Reshetikhin--Turaev $SO_3$-Topological Quantum Field Theory (TQFT), where  the TQFT is well-known to
be closely related to the desired Turaev--Viro invariants. In particular, our utilized method relies
on the invertibility of the linear operator for which we provide necessary and sufficient conditions.
\end{abstract}

\newtheorem{innercustomgeneric}{\customgenericname}
\providecommand{\customgenericname}{}
\newcommand{\newcustomtheorem}[2]{%
  \newenvironment{#1}[1]
  {%
   \renewcommand\customgenericname{#2}%
   \renewcommand\theinnercustomgeneric{##1}%
   \innercustomgeneric
  } {\endinnercustomgeneric} }

\newcustomtheorem{customthm}{Theorem}
\newcustomtheorem{customlemma}{Lemma}
\newcustomtheorem{customprop}{Proposition}
\newcustomtheorem{customconjecture}{Conjecture}
\newcustomtheorem{customcor}{Corollary}

\theoremstyle{plain}
\newtheorem*{ack*}{Acknowledgements}
\newtheorem{thm}{Theorem}[section]
\newtheorem{lem}[thm]{Lemma}
\newtheorem{prop}[thm]{Proposition}
\newtheorem{cor}[thm]{Corollary}
\newtheorem{predefinition}[thm]{Definition}
\newtheorem{conjecture}[thm]{Conjecture}
\newtheorem{preremark}[thm]{Remark}
\newenvironment{remark}%
  {\begin{preremark}\upshape}{\end{preremark}} \newenvironment{definition}%
  {\begin{predefinition}\upshape}{\end{predefinition}}

\newtheorem{ex}[thm]{Example}
\newtheorem{ques}[thm]{Question}



\section{Introduction}\label{IntroSec} For a compact 3-manifold $M$, its Turaev--Viro invariants are
a family of $\mathbb{R}$-valued homeomorphism invariants parameterized by an integer $r\geq 3$
depending on a $2r$-th root of unity $q$. We are primarily interested in the invariants when $r$ is
odd and $q = e^{\frac{2\pi \sqrt{-1}}{r}}$. 

In this paper, we study the variation of the Turaev--Viro invariants of a $3$-manifold with toroidal
boundary when we attach a $(p,q)$-cable space.

\begin{definition}
Let $V$ be the standardly embedded solid torus in $S^3$, and let $V'$ be a closed neighborhood of
$V$. For $p$ and $q$ coprime integers with $q>0$, let $T_{p,q} \subset \partial V$ be the torus knot of
slope $p/q$. The $(p,q)$-cable space, denoted $C_{p,q}$, is the complement of the torus knot
$T_{p,q}$ in $V'$. Let $M$ be a $3$-manifold with toroidal boundary. A manifold $M'$ obtained from
gluing a $(p,q)$-cable space $C_{p,q}$ to a boundary component of $M$ along the exterior toroidal
boundary component of $C_{p,q}$ is called a $(p,q)$-cable of $M$.
\end{definition}

Our main theorem is the following.

\begin{thm}\label{mainthm} Let $M$ be a manifold with toroidal boundary, let $p,q$ be coprime
integers with $q>0$, and let $r \geq 3$ be an odd integer coprime to $q$. Suppose $M'$ is a
$(p,q)$-cable of $M$. Then there exists a constant $C>0$ and natural number $N$ such that
\[
\frac{1}{Cr^N}TV_r(M) \leq TV_r(M') \leq Cr^N TV_r(M).
\]
\end{thm}

Theorem \ref{mainthm} has notable applications to existing conjectures. In general, the Turaev--Viro
invariants are difficult to compute; however, there is interest in the relationship between their
$r$-asymptotic behavior and classical invariants of 3-manifolds. Chen and Yang \cite{chen2018volume}
conjectured that the growth rate of the Turaev--Viro invariants for hyperbolic manifolds is related to the
manifold's hyperbolic volume. They also provided computational evidence for the conjecture in
\cite{chen2018volume}. 

This conjecture should be compared to the well-known conjectures of Kashaev \cite{KashaevVolConj}
and Murakami-Murakami \cite{MurakamiMurakamiVolConj} relating the Kashaev and colored Jones
invariants of hyperbolic link complements to their hyperbolic volumes. Detcherry and Kalfagianni
\cite{detcherry2019gromov} restated the Turaev--Viro invariant volume conjecture more generally in
terms of the simplicial volume for manifolds which are not necessarily hyperbolic. In order to state
the conjecture, we will first introduce a slightly weaker condition for the growth rate of the
Turaev--Viro invariants.

\begin{definition}\label{Def:Lim} Define the following two asymptotics of the Turaev--Viro
invariants for compact 3-manifolds as
\begin{align*}
\textit{lTV}(M) := \liminf_{r\rightarrow \infty, \text{ } r \text{ odd}} \frac{2\pi}{r} &\log \left|\text{TV}_r \left(M; q = e^{\frac{2\pi i}{r}}\right)\right|,
\end{align*}
and
\begin{align*}
\textit{LTV}(M) := \limsup_{r\rightarrow \infty, \text{ } r \text{ odd}} \frac{2\pi}{r} &\log \left|\text{TV}_r \left(M; q = e^{\frac{2\pi i}{r}}\right)\right|.
\end{align*}
\end{definition}

Additionally, we will introduce the simplicial volume for compact orientable $3$-manifolds with
empty or toroidal boundary, originally defined by Gromov \cite{Gromov}. For $M$ a compact,
orientable, irreducible $3$-manifold, there is a unique collection of incompressible tori, up to
isotopy, along which $M$ can be decomposed into atoroidal manifolds. This is known as the JSJ
decomposition \cite{JacoShalen,Johannson}. By Thurston's Geometrization Conjecture
\cite{ThurstonGeomConj}, famously completed in the work of Perelman
\cite{PerelmanEntropy,PerelmanFinite,PerelmanRicci}, each of these atoroidal manifolds are either
hyperbolic or Seifert-fibered, and by Thurston \cite{ThurstonGT3manifolds}, the simplicial volume of
$M$ coincides with the sum of the simplicial volumes of the resulting pieces. In the case where $M$
is hyperbolic, the simplicial volume is positive and is related to the hyperbolic volume of $M$ by 
$$vol(M)=v_3 \|M\|$$ where $vol(M)$ is the hyperbolic volume of $M$, $v_3 \approx 1.0149$ is the
volume of the regular ideal tetrahedron, and $\|M\|$ is the simplicial volume of $M$. 

This leads to a natural extension of Chen and Yang's Turaev--Viro invariant volume conjecture
\cite{chen2018volume}.
\begin{conjecture}[\cite{chen2018volume}, \cite{detcherry2019gromov}]\label{GenCYvol} Let $M$ be a
compact oriented 3-manifold. Then
\[
\text{LTV}(M) = v_3||M||,
\]
where $v_3$ is the volume of the regular ideal tetrahedron and $\|\cdot\|$ is the simplicial volume. 
\end{conjecture}

If Conjecture \ref{GenCYvol} is true, since the simplicial volume is additive under the JSJ
decomposition, it would imply that the asymptotics of the Turaev--Viro invariants are also additive
under the decomposition. As an application of Theorem \ref{mainthm}, we provide evidence of the
additivity for the asymptotics of the Turaev--Viro invariants. 

The asymptotic additivity property has been explored in several works. For a manifold $M$ which
satisfies Conjecture \ref{GenCYvol}, the property was proven for invertible cablings of $M$ by
Detcherry and Kalfagianni \cite{detcherry2019gromov}, the figure-eight knot cabled with Whitehead
chains by Wong \cite{wong2020WHFig8}, and an infinite family of manifolds with arbitrarily large
simplicial volume by the authors of this paper \cite{kumarMelby21}. Additionally, the property was
proven to hold under the operation of attaching a $(p, 2)$-cable space by Detcherry
\cite{DetcherryCabling}, which we extend in this work. 

A key property of the $(p,q)$-cable spaces is that they have simplicial volume zero. Theorem
\ref{mainthm} provides a way to construct new manifolds without changing the simplicial volume while
controlling the growth of the Turaev--Viro invariants. This leads to many examples of manifolds
satisfying Conjecture \ref{GenCYvol}. 

To the authors' knowledge, in all of the proven examples of Conjecture \ref{GenCYvol}, the stronger
condition that the limit approaches the simplicial volume is verified, as opposed to  only the limit
superior. Theorem \ref{mainthm} implies the following corollaries. See Section \ref{Sec: BoundingCables} for more details.

\begin{cor}\label{Weakequivlim} Suppose $M$ satisfies Conjecture \ref{GenCYvol} and $lTV(M) =
v_3||M||$. Then for any $p$ and $q$ coprime, any $(p,q)$-cable $M'$ also satisfies Conjecture
\ref{GenCYvol}.
\end{cor}

Some examples which satisfy the hypothesis of Corollary \ref{Weakequivlim} include the figure-eight
knot and the Borromean rings by Detcherry-Kalfagianni-Yang \cite{colJvolDKY}, the Whitehead chains
by Wong \cite{WongWhitehead}, the fundamental shadow links by Belletti-Detcherry-Kalfagianni-Yang
\cite{growth6j}, a family of hyperbolic links in $S^2\times S^1$ by Belletti \cite{Belletti2021}, a
large family of octahedral links in $S^3$ by the first author of this paper \cite{Kumar21}, and a
family of link complements in trivial $S^1$-bundles over oriented connected closed surfaces by the
authors of this paper \cite{kumarMelby21}.

For general $p$ and $q$ coprime, Corollary \ref{Weakequivlim} demonstrates the stability of Conjecture \ref{GenCYvol} under the $(p,q)$-cabling operation.
However, in the case when $q=2^n$
for $n\in\mathbb{N}$, we recover the full limit as shown in the following corollary. 

\begin{cor}\label{Cor:2n} Suppose $M$ satisfies Conjecture \ref{GenCYvol}. Then for any odd $p$ and
$n \in \mathbb{N}$, any $(p,2^n)$-cable $M'$ also satisfies Conjecture \ref{GenCYvol}. Moreover, if
$lTV(M) = v_3||M||$, then $lTV(M') =LTV(M') = v_3 ||M'||$.
\end{cor}
As a direct result of  Corollaries \ref{Weakequivlim} and \ref{Cor:2n}, we extend the work of
Detcherry \cite{DetcherryCabling}, where the author considers the operation of attaching a
$(p,2)$-cable space. This allows us to construct manifolds satisfying Conjecture \ref{GenCYvol} from
manifolds with toroidal boundary where the conjecture is already known. This includes all previously
mentioned examples. 

The general method of proof for Theorem \ref{mainthm} follows from the work of Detcherry
\cite{DetcherryCabling}. Considering the cable space $C_{p,q}$ as a cobordism between tori, the
Reshetikhin--Turaev $SO_3$-TQFT at level $r$, denoted by $RT_r$, associates to it a linear operator
$$RT_r(C_{p,q}): RT_r(T^2) \rightarrow RT_r(T^2).$$

For $p$ odd and $q=2$, Detcherry presents $RT_r(C_{p,2})$ using the basis $\{e_1,e_3,\dots,
e_{2m-1}\}$, which is equivalent to the orthonormal basis $\{e_1,e_2,\dots, e_{m}\}$ for $RT_r(T^2)$
given in \cite{BHMVKauffman} under the symmetry $e_{m-i} = e_{m+i+1}$ for $0 \leq i\leq m-1$. More
details of the construction are given in Section \ref{Sec:Preliminaries}. With this basis,
$RT_r(C_{p,2})$ can be presented as a product of two diagonal matrices and one triangular matrix.
This allows the author to directly write the inverse of $RT_r(C_{p,2})$. From the inverse of this linear operator, Detcherry establishes a lower bound of the Turaev--Viro invariants under
attaching a $(p,2)$-cable space. 

For general $q$, $RT_r(C_{p,q})$ does not have as simple a presentation under the same basis, making
it more difficult to conclude that $RT_r(C_{p,q})$ is invertible. In order to resolve this, we
present $RT_r(C_{p,q})$ using a different basis for $RT_r(T^2)$, defined in Section \ref{Sec:
TechLemmaPfs}, that allows us to also show directly that  $RT_r(C_{p,q})$ is invertible provided $r$
and $q$ are coprime. Following Detcherry's argument, the invertibility of $RT_r(C_{p,q})$ is
integral in finding the lower bound from Theorem \ref{mainthm}; however, the invertibility of
$RT_r(C_{p,q})$ is constrained by the condition that $r$ and $q$ are coprime, as outlined by Theorem
\ref{RTinvertible}.

\begin{thm}\label{RTinvertible} Let $p$ be coprime to some positive integer $q$. Then
    $RT_r(C_{p,q})$ is invertible if and only if $r$ and $q$ are coprime. Moreover, the operator norm $|||RT_r(C_{p,q})^{-1}|||$ grows at most polynomially.
\end{thm}

As we will show in Section \ref{Sec: BoundingCables}, the coprime condition between $r$ and $q$
leads to the discrepancy between recovering the limit superior in Corollary \ref{Weakequivlim}
versus the full limit in Corollary \ref{Cor:2n}.

The paper is organized as follows: We recall properties of the Reshetikhin--Turaev $SO_3$-TQFTs, the
$RT_r$ torus knot cabling formula, and relevant properties of the Turaev--Viro invariants in Section
\ref{Sec:Preliminaries}. In Section \ref{Sec: BoundingCables}, we prove Theorem \ref{mainthm}
assuming Theorem \ref{RTinvertible}. In Section \ref{Sec: TechLemmaPfs}, the construction of the
relevant basis for $RT_r(T^2)$ and the proof of Theorem \ref{RTinvertible} is given. Lastly, we
consider future directions in Section \ref{Sec: FD}. 

\begin{ack*}
The authors would like to thank their advisor Efstratia Kalfagianni for guidance and helpful
discussions. Additionally, we would like to thank Renaud Detcherry  for his comments and suggestions
on an early version of this paper. 
Finally, we thank the reviewer for their time and careful reading as well as their valuable suggestions for improving the overall clarity of the arguments in Section 4. 
Part of this research was supported in the form of graduate
Research Assistantships by NSF Grants DMS-1708249 and DMS-2004155. For the second author, this
research was partially supported by a Herbert T. Graham Scholarship through the Department of
Mathematics at Michigan State University. 
\end{ack*}

\section{Preliminaries}\label{Sec:Preliminaries}

\subsection{Reshetikhin--Turaev TQFT}
In this subsection, we outline relevant properties of the Reshetikhin--Turaev $SO_3$-TQFTs, which
were defined by Reshetikhin and Turaev in \cite{ReshetikhinTuraev1991}. Let $\mathfrak{Cob}$ be the
category of $(2+1)$-dimensional cobordisms and $Vect(\mathbb{C})$ be the category of
$\mathbb{C}$-vector spaces. For an odd integer $r\geq 3$ and primitive $2r$-th root of unity $A$,
one associates a $(2+1)$-dimensional TQFT $RT_r:\mathfrak{Cob} \rightarrow Vect(\mathbb{C})$.
Blanchet, Habegger, Masbaum, and Vogel \cite{BHMVKauffman} gave a skein-theoretic framework for this
$SO_3$-TQFT, and its main properties are the following: 

\begin{enumerate}[1)]
    \item For a closed oriented surface $\Sigma$, $RT_r(\Sigma)$ is a finite dimensional vector
    space over $\mathbb{C}$ with the natural Hermitian form. For a disjoint union $\Sigma \sqcup
    \Sigma'$, we have $RT_r(\Sigma \sqcup \Sigma') = RT_r(\Sigma) \otimes RT_r(\Sigma')$.
    \item For an oriented closed 3-manifold $M$, $RT_r(M) \in \mathbb{C}$ is a topological
    invariant. 
    \item For an oriented compact 3-manifold $M$ with boundary $\partial M$, $RT_r(M)\in
    RT_r(\partial M)$ is a vector.
    \item For a cobordism $(M,\Sigma_1, \Sigma_2)$, $RT_r(M): RT_r(\Sigma_1) \rightarrow
    RT_r(\Sigma_2)$ is a linear map.
\end{enumerate}

In \cite{BHMVKauffman}, the authors also give explicit bases for any surface. However, we will focus on
$RT_r(T^2)$, which can be considered as a quotient of the Kauffman skein module of the genus $1$
handlebody $D^2 \times S^1$. 

We begin by coloring the core $\{0\}\times S^1$ by the $(i-1)$-th Jones-Wenzl idempotent. This gives
a family of elements $e_i$ of the Kauffman skein module of the solid torus. However, there are only
finitely many Jones-Wenzl idempotents for a given odd $r = 2m+1$ and $2r$-th root of unity $A$,
namely $e_1,\dots,e_{2m-1}$ \cite{BHMVKauffman}. We can consider these $e_i$'s as elements of the
quotient $RT_r(T^2)$, giving us a basis.

\begin{thm}[\cite{BHMVKauffman}, Theorem 4.10]
Let $r=2m+1 \geq 3$. Then the family $e_1,\dots,e_m$ is an orthonormal basis for $RT_r(T^2)$.
Moreover, the relation $e_{m-i} = e_{m+1+i}$ holds for $0\leq i \leq m-1$.
\end{thm}
The second part of the theorem implies that $\{e_1,e_3,\dots ,e_{2m-1}\}$ is just a reordering of
the basis $\{e_1,\dots, e_m\}$.

\subsection{The Cabling Formula}
Here, we will give an explicit description for the Reshetikhin--Turaev invariants of the torus knot
cable spaces.

Let $p$ and $q$ be coprime integers where $q>0$, and let $C_{p,q}$ be the $(p,q)$-cable space $C_{p,q}$.
These spaces are Seifert-fibered and therefore have simplicial volume zero. For $r=2m+1 \geq 3$, we
extend the vectors $e_i \in RT_r(T^2)$ to all $i \in \mathbb{Z}$ in the following way. Let $e_{-i} =
-e_i$ for any $i \geq 0$, and let $e_{i+kr} = (-1)^ke_i$ for any $k \in \mathbb{Z}$. Note this means
that $e_r = e_0 = 0$.

Regarding the cable space $C_{p,q}$ as a cobordism between tori, the Reshetikhin--Turaev $SO_3$-TQFT
gives a linear map
\[
RT_r(C_{p,q}): RT_r(T^2) \rightarrow RT_r(T^2).
\]
The map $RT_r(C_{p,q})$ sends the element $e_i$ to the element of $RT_r(T^2)$ corresponding to a
$(p,q)$-torus knot embedded in the solid torus and colored by the $(i-1)$-th Jones-Wenzl idempotent.
Morton \cite{MortonCabling} gives the following formula for the image of the basis elements under
$RT_r(C_{p,q})$.

\begin{thm}[\cite{MortonCabling}, Section 3, Cabling Formula] \label{CablingFormula}
\[
RT_r(C_{p,q})(e_i) = A^{pq(i^2-1)/2} \sum_{k \in S_i} A^{-2pk(qk +1)}e_{2qk+1},
\]
where $S_i$ is the set
\[
S_i = \{-\frac{i-1}{2}, -\frac{i-3}{2}, ..., \frac{i-3}{2}, \frac{i-1}{2}\}.
\]
\end{thm}

As we will see in Subsection \ref{Subsec: TVProp}, the Reshetikhin--Turaev TQFT is closely related
to the Turaev--Viro invariants for $3$-manifolds. By using their relationship, the explicit formula
given in Theorem \ref{CablingFormula} will allow us to obtain a lower bound on the Turaev--Viro
invariants under the cabling operation. 

\subsection{Properties of the Turaev--Viro invariants}\label{Subsec: TVProp} In this subsection, we
discuss properties of the Turaev--Viro invariants \cite{TuraevViro} as well as an important
characterization in terms of the Reshetikhin--Turaev invariants. 

The Turaev--Viro invariants were defined by Turaev and Viro \cite{TuraevViro} in terms of state sums
over triangulations of a 3-manifold $M$, but they are also closely related to the
Reshetikhin--Turaev invariants. The following identity was originally proven for closed 3-manifolds
by Roberts \cite{RobertsSkein} and then extended to compact manifolds with boundary by Benedetti and
Petronio \cite{TVCompact}.

\begin{thm}[\cite{TVCompact, RobertsSkein}]\label{TVandRT} Let $r\geq3$ be an odd integer, and let
$q$ be a primitive $2r$-th root of unity. Then for a compact oriented manifold $M$ with toroidal
boundary, 
$$TV_r\left(M;q\right) =\left \Vert RT_r\left(M;q^\frac{1}{2}\right) \right \Vert^2$$ where $\Vert
\cdot \Vert$ is the natural Hermitian norm on $RT_r\left(\partial M\right).$
\end{thm}

We note that this identity holds more generally, but we have restricted to manifolds with toroidal
boundary for simplicity.

In \cite{detcherry2019gromov}, Detcherry and Kalfagianni proved that the growth rate of the
Turaev--Viro invariants has properties reminiscent of simplicial volume. We summarize their results
in the following theorem.

\begin{thm}[\cite{detcherry2019gromov}]\label{TVProperties} Let $M$ be a compact oriented
$3$-manifold, with empty or toroidal boundary.
\begin{enumerate}[1)]
    \item If $M$ is a Seifert manifold, then there exist constants $B>0$ and $N$ such that for any
    odd $r\geq 3$, we have $TV_r(M) \leq Br^N$ and $LTV(M)\leq 0$.
    \item If $M$ is a Dehn-filling of $M'$, then $TV_r(M) \leq TV_r(M')$ and $LTV(M) \leq LTV(M')$.
    \item If $M=M_1 \underset{T}{\bigcup} M_2$ is obtained by gluing two $3$-manifolds along a torus
    boundary component, then $TV_r(M)\leq TV_r(M_1)TV_r(M_2)$ and $LTV(M) \leq LTV(M_1)+LTV(M_2)$.
    \end{enumerate}
\end{thm}


\section{Bounding the Invariant Under Cabling}\label{Sec: BoundingCables}

In this section, we will prove Theorem \ref{mainthm} with the assumption of a key theorem, and we
reserve the technical details for Section \ref{Sec: TechLemmaPfs}.  We remark that the major
components of our argument follow from the work of Detcherry \cite{DetcherryCabling} where the case
when $p$ is odd and $q=2$ was proven. For convenience, we  will restate the main theorem.
\begin{customthm}{\ref{mainthm}} Let $M$ be a manifold with toroidal boundary, let $p,q$ be coprime
integers with $q>0$, and let $r \geq 3$ be an odd integer coprime to $q$. Suppose $M'$ is a
$(p,q)$-cable of $M$. Then there exists a constant $C>0$ and natural number $N$ such that
\[
\frac{1}{Cr^N}TV_r(M) \leq TV_r(M') \leq Cr^N TV_r(M).
\]
\end{customthm}
We will now assume Theorem \ref{RTinvertible}, which we also restate for convenience. 

\begin{customthm}{\ref{RTinvertible}} Let $p$ be coprime to some positive integer $q$. Then
    $RT_r(C_{p,q})$ is invertible if and only if $r$ and $q$ are coprime. Moreover, the operator norm $|||RT_r(C_{p,q})^{-1}|||$ grows at most polynomially.
\end{customthm}

\begin{proof}[Proof of Theorem \ref{mainthm}]
As mentioned previously, the case when $p$ is odd and $q=2$ was shown by Detcherry
\cite{DetcherryCabling}, and our approach follows closely in structure. We let $M$ be a manifold
with toroidal boundary, $p$ an integer, $q>0$ an integer coprime to $p$, $r \geq 3$ odd and coprime
to $q$, and $M'$ a $(p,q)$-cable of $M$.  We will proceed to prove Theorem \ref{mainthm} by showing
the upper  inequality of
\[
\frac{1}{Cr^N}TV_r(M) \leq TV_r(M') \leq Cr^N TV_r(M)
\]
followed by the lower inequality, where  $C>0$ and $N \in \mathbb{N}$. To obtain the upper
inequality, we first remark that $M'=C_{p,q}\underset{T}{\bigcup} M$. By Theorem \ref{TVProperties},
this implies that 
$$TV_r(M') \leq TV_r(C_{p,q}) TV_r(M).$$ Since $C_{p,q}$ is a Seifert manifold, we have that 
$$TV_r(C_{p,q}) \leq C_1r^{N_1}$$ for some $C_1>0$ and $N_1\in\mathbb{N}$ also by Theorem
\ref{TVProperties}. This leads to the upper inequality 
$$TV_r(M') \leq C_1r^{N_1} TV_r(M).$$ For the lower inequality, we will use Theorem
\ref{RTinvertible}. From the properties of the Reshetikhin--Turaev $SO_3$-TQFT, we consider the
linear map
\[
RT_r(C_{p,q}): RT_r(T^2) \rightarrow RT_r(T^2).
\]
If $M$ only has one boundary component, then  
$$RT_r(M')=RT_r(C_{p,q})RT_r(M)$$ by the properties of a TQFT. If $M$ has other boundary components,
then the invariant associated to any coloring $i$ of the other boundary components may be computed
as 
$$RT_r(M',i)=RT_r(C_{p,q})RT_r(M,i).$$ By the invertibility of $RT_r(C_{p,q})$ from Theorem
\ref{RTinvertible}, we have the inequality 
$$||RT_r(M)|| \leq |||{RT_r(C_{p,q})^{-1}}|||\cdot || RT_r(M')||$$ where  $||\cdot||$ is the norm
induced by the Hermitian form of the TQFT and $|||\cdot|||$ is the operator norm. Since the operator
norm grows at most polynomially by Theorem \ref{RTinvertible}, we obtain the inequality 
$$\frac{1}{C_2r^{N_2}}||RT_r(M)|| \leq ||RT_r(M')|| $$ for some $C_2 >0$ and $N_2\in \mathbb{N}$.
Lastly, by Theorem \ref{TVandRT}, the norm of the Reshetikhin--Turaev invariant is related to the
Turaev--Viro invariant such that we arrive to the desired inequality $$\frac{1}{C_3r^{N_3}}TV_r(M)
\leq TV_r(M') $$ for some $C_3 >0$ and $N_3\in \mathbb{N}$. This leads to 
\[
\frac{1}{Cr^N}TV_r(M) \leq TV_r(M') \leq Cr^N TV_r(M)
\]
 where  $C>0$ and $N \in \mathbb{N}$.
\end{proof}

As discussed in Section \ref{IntroSec}, the following corollaries follow from Theorem \ref{mainthm}.

\begin{customcor}{\ref{Weakequivlim}} Suppose $M$ satisfies Conjecture \ref{GenCYvol} and $lTV(M) =
v_3||M||$. Then for any $p$ and $q$ coprime, any $(p,q)$-cable $M'$ also satisfies Conjecture
\ref{GenCYvol}.
\end{customcor}

\begin{proof}
By Theorem \ref{TVProperties} Part $(1)$, $LTV(C_{p,q}) \leq 0$, and thus by Theorem
\ref{TVProperties} Part $(3)$, $LTV(M') \leq LTV(M)$. Since $lTV(M)=LTV(M)=v_3\|M\|$, the limit
exists, and any subsequence also converges to $v_3\|M\|$. By Theorem \ref{mainthm} along odd $r$,
\begin{align*}
    \limsup_{r\rightarrow \infty, \text{ } (r,q)=1} \frac{2\pi}{r} \log \left|\text{TV}_r \left(M'\right)\right| &=\limsup_{r\rightarrow \infty, \text{ } (r,q)=1} \frac{2\pi}{r} \log \left|\text{TV}_r \left(M\right)\right| = LTV(M) = v_3||M||,
\end{align*}
where $$\limsup_{r\rightarrow \infty, \text{ } (r,q)=1} \frac{2\pi}{r} \log \left|\text{TV}_r
\left(-\right)\right|$$ is the limit superior of the subsequence along which $r$ and $q$ are
coprime.

Since
\begin{align*}
         v_3||M||=\limsup_{r\rightarrow \infty, \text{ } (r,q)=1} \frac{2\pi}{r} \log \left|\text{TV}_r \left(M'\right)\right| & \leq LTV(M') \leq LTV(M)=v_3||M||,
\end{align*}
we have $$LTV(M') = v_3 ||M|| = v_3 ||M'||,$$ where the final equality follows from the fact that
the simplicial volume does not change under attaching a $(p,q)$-cable space.
\end{proof}

\begin{customcor}{\ref{Cor:2n}} Suppose $M$ satisfies Conjecture \ref{GenCYvol}. Then for any odd
$p$ and $n \in \mathbb{N}$, any $(p,2^n)$-cable $M'$ also satisfies Conjecture \ref{GenCYvol}.
Moreover, if $lTV(M) = v_3||M||$, then $lTV(M') =LTV(M') = v_3 ||M'||$.
\end{customcor}

\begin{proof}
Since $r$ is odd, $(r,2^n) = 1$ for any $n\geq 1$, which means Theorem \ref{mainthm} holds for any
$(p,2^n)$-cable of $M$ provided $p$ is odd. Since $||M||=||M'||$, this implies that
$$LTV(M') = LTV(M) = v_3||M||= v_3||M'||,$$ so $M'$ also satisfies Conjecture \ref{GenCYvol}.

Theorem \ref{mainthm} also implies that $lTV(M') = lTV(M)$. In the case where $lTV(M) = v_3 ||M||$,
we recover the full limit $$lTV(M')=lTV(M)= v_3||M| = v_3||M'|| = LTV(M').$$
\end{proof}


\section{Proof of Supporting Theorem}\label{Sec: TechLemmaPfs}

In this section, we will provide a proof of Theorem \ref{RTinvertible}, which we restate here for
convenience.

\begin{customthm}{\ref{RTinvertible}} Let $p$ be coprime to some positive integer $q$. Then
    $RT_r(C_{p,q})$ is invertible if and only if $r$ and $q$ are coprime. Moreover, the operator norm $|||RT_r(C_{p,q})^{-1}|||$ grows at most polynomially.
\end{customthm}

We will use the following supporting proposition for the proof of Theorem \ref{RTinvertible}, which
is given in Subsection \ref{Subsec: RTinvertible}. The proof of this proposition is given in
Subsection \ref{subsec: suppProp}. We also use a couple of technical lemmas which are subsequently
proven in Subsection \ref{subsec: TechLemmaPfs}. We begin by constructing a basis over which
$RT_r(C_{p,q})$ admits a simpler expression.

By the cabling formula given by Theorem \ref{CablingFormula}, 
\begin{align*}
    RT_r(C_{p,q})(e_i) \in Span\{e_1, e_{ql+1}, e_{ql-1}\}_{l=1}^{m-1}
\end{align*}
where $m=\frac{r-1}{2}$. Let $F_m := \{f_l\}_{l=0}^{m-1}$, where
\begin{align*}
    f_0 := & e_1  \\
    f_l := & e_{ql+1}-A^{2pl} e_{ql-1} \qquad l=1,\dots, m.
\end{align*}

Define $\Tilde{f}_l \in Span\{e_1, \dots, e_m\}$ to be the reduction of $f_l$ under the quotient
induced by the symmetries $e_{-i} = -e_i$ for any $i \geq 0$ and $e_{i+nr} = (-1)^ne_i$ for any $n
\in \mathbb{Z}$. Note that for each $l$, $ql\pm 1 = kr+j$ for some non-negative integers $k, j$
where $0 \leq j <r$. This means that up to sign, these symmetries imply 
\begin{align}
    e_{ql\pm 1} &= e_{ql-kr\pm 1} = e_j && \text{for  } 0 \leq j\leq m \label{eRed1}\\
    e_{ql\pm 1} &= e_{(k+1)r-ql\mp 1} = e_{r-j} && \text{for  } m+1 \leq j< r. \label{eRed2}
\end{align}

Finally, define $\Tilde{F}_m := \{\Tilde{f}_l\}_{l=0}^{m-1}$, and let $R_m$ be the $(m\times
m)$-matrix with columns corresponding to the reduced vectors $\Tilde{f}_l$, for $l=0,\dots, m-1$. In
particular, $\Tilde{f}_l$ corresponds to $col(l+1)$ of $R_m$, and the rows of $R_m$ correspond to
the original orthonormal basis $\{e_1, \dots, e_m\}$ spanning $RT_r(T^2)$.

\begin{remark}
We note that $F_m$, $\Tilde{F}_m$, and $R_m$ are also dependent on $p$ and $q$, but these
dependencies are suppressed to avoid unwieldy notation. 
\end{remark}

The following proposition will be used to prove Theorem \ref{RTinvertible}.

\begin{prop}\label{2NBasisSpans} Let $r = 2m+1 \geq3$ be coprime to $q$. Then $R_m$ is a change of basis
from $\Tilde{F}_m \rightarrow \{e_1, \dots, e_m\}$ and the operator norm $|||R_m^{-1}|||$ grows at
most polynomially in $m$. Moreover, for $i \in \{1,\dots, m\}$,

\begin{align}\label{Coeffofl}
RT_r(C_{p,q})(e_i) =& A^{\frac{qp}{2}(i^2-1)} \sum_{l \in T_i} A^{-p\left(\frac{q}{2}l^2 + l\right)} \Tilde{f}_l,
\end{align}
where $T_i = \{0,2,\dots, i-1\}$ for odd $i$ and $T_i = \{1,3,\dots, i-1\}$ for even $i$.
\end{prop}

The idea of the proof is to leverage symmetric properties of the $\Tilde{f}_l$ to give a
presentation of $R_m^{-1}$ and bound its operator norm. The assumption that $(r,q)=1$ is necessary
for invertibility, as indicated by the following proposition.

\begin{prop}\label{RNotInvert} Suppose $r =2m+1 \geq 3$ is odd and not coprime to $q$. Then $R_m$ is
singular.
\end{prop}

The proofs of Propositions \ref{2NBasisSpans} and \ref{RNotInvert} will be given in Subsection
\ref{subsec: suppProp}. 

\subsection{Proof of Theorem \ref{RTinvertible}}\label{Subsec: RTinvertible}

We now can proceed with the proof of Theorem \ref{RTinvertible} assuming Proposition
\ref{2NBasisSpans}.

\begin{proof}[Proof of Theorem \ref{RTinvertible}]
We begin with the necessary condition. Suppose $(r,q) = d>1$. Then there are coprime $q',r'$ such
that $q=dq'$ and $r=dr'$. We claim that $row(nd)$ of $RT_r(C_{p,q})$ consists of only zeros for each
$n$. Suppose some $e_{ql\pm 1} = e_{kr+j}$, where $0 \leq j \leq m$, reduces to $e_j = e_{nd}$. Then
by Equation (\ref{eRed1}), $ql-kr\pm 1 = nd$, which means $d(q'l-r'k-n) = \mp 1$, which is a
contradiction. Similarly, if $e_{ql\pm 1} = e_{kr+j}$, where $m+1 \leq j <r$, reduces to $e_{r-j} =
e_{nd}$. Then by Equation (\ref{eRed2}), $d((1+k)r'-q'l-n) = \pm 1$, which is also a contradiction.
This means that $row(nd) = [0,\dots, 0]$, thus $RT_r(C_{p,q})$ is singular.

For sufficiency, suppose $(r,q) = 1$. By Proposition \ref{2NBasisSpans}, we can write
$RT_r(C_{p,q})$ as a product of two diagonal matrices with an upper-triangular matrix and the change
of basis $R_m$:

\begin{align*}
    RT_r(C_{p,q}) = & R_m \left( \begin{array}{cccc}
    1 & 0 & \dots & 0 \\
    0 & A^{-p\left((2-1)+\frac{q}{2}(2-1)^2\right)} & \ddots & \vdots \\
    \vdots & \ddots & \ddots & 0 \\
    0 & \dots & 0 & A^{-p\left((m-1)+\frac{q}{2}(m-1)^2\right)}
    \end{array} \right) \times \\
    &\left( \begin{array}{ccccccc}
    1 & 0 & 1 & 0 & 1 & 0 & \dots\\
    0 & 1 & 0 & 1 & 0 & 1 & \dots \\
    & & 1 & 0 & 1 & 0 & \dots\\
    \vdots & & & 1 & 0 & 1 & \dots\\
    & & & & 1 & 0 & \\
    & & & & & \ddots  & \\
    0 & & \dots & & & 0 & 1
    \end{array} \right)
    \left( \begin{array}{cccc}
    1 & 0 & \dots & 0 \\
    0 & A^{\frac{qp}{2}(2^2-1)} & \ddots & \vdots \\
    \vdots & \ddots & \ddots & 0 \\
    0 & \dots & 0 & A^{\frac{qp}{2}(m^2-1)}
    \end{array} \right).
\end{align*}
Note that the columns of the middle upper triangular matrix correspond to the index sets $T_i$ of
the sum in Equation (\ref{Coeffofl}). Inverting this product, we have
\begin{align*}
   RT_r(C_{p,q})^{-1} =& \left( \begin{array}{cccc}
    1 & 0 & \dots & 0 \\
    0 & A^{\frac{qp}{2}(1-2^2)} & \ddots & \vdots \\
    \vdots & \ddots & \ddots & 0 \\
    0 & \dots & 0 & A^{\frac{qp}{2}(1-m^2)}
    \end{array} \right)
    \left( \begin{array}{ccccccc}
    1 & 0 & -1 & 0 & 0 &\dots & 0\\
    0 & 1 & 0 & -1 & 0 & \dots &0\\
     & & 1 & 0 & -1 & 0& \vdots\\
    \vdots & & & 1 & 0 & -1 \dots & 0\\
     & & & & 1 & \ddots & -1\\
     & & & & & \ddots & 0\\
    0 & & &\dots  & & 0 & 1
    \end{array} \right) \\
    & \times \left( \begin{array}{cccc}
    1 & 0 & \dots & 0 \\
    0 & A^{p\left((2-1)+\frac{q}{2}(2-1)^2\right)} & \ddots & \vdots \\
    \vdots & \ddots & \ddots & 0 \\
    0 & \dots & 0 & A^{p\left((m-1)+\frac{q}{2}(m-1)^2\right)}
    \end{array} \right) R_m^{-1}.
\end{align*}
By Proposition \ref{2NBasisSpans}, $|||R_m^{-1}|||$ grows at most polynomially in $m$, so it is
bounded polynomially in $r$. For the total bound, observe that both of the diagonal matrices are
isometries, and the upper triangular matrix has operator norm bounded above by a polynomial in $r$
by the Cauchy-Schwartz inequality 
\end{proof}

\subsection{Proof of Propositions \ref{2NBasisSpans} and \ref{RNotInvert}}\label{subsec: suppProp}

We give proofs of Propositions \ref{2NBasisSpans} and
\ref{RNotInvert} in this subsection. The following definitions and lemmas will be useful in the proofs.

By the symmetries $e_{-i} = -e_i$ for any $i \geq 0$ and $e_{i+kr} = (-1)^ke_i$ for any $k \in
\mathbb{Z}$, we may extend the definition of $f_l$ to all $l\in \mathbb{Z}$ using the following symmetries:
\begin{itemize}
    \item $f_l = e_{1+ql} + A^{2pl} e_{1-ql}$ for any $l \in \mathbb{Z}$,
    \item $f_{l+r} = (-1)^qf_l$ for any $l \in \mathbb{Z}$, and 
    \item $f_l = A^{2pl} f_{-l}$.
\end{itemize}

The following Lemma will be used to present $R_m^{-1}$.

\begin{lem}\label{e_i_as_f_ls} Let $r = 2m+1 \geq 3$ be coprime to $q$, and let $q^*$ be the inverse
    of $q$ modulo $r$. Then for $l \in \{0, \dots, m-1 \}$,  
    \begin{equation}\label{e_as_f_eqn}
        e_{l+1} = 
        \begin{cases}
           f_0 & \text{ if } l=0 \\
          f_{q^*} & \text{ if } l=1 \\ f_{q^*l}+\sum_{k=1}^{\lfloor l / 2 \rfloor} A^{2pq^*\left(kl- \sum_{i=0}^{k-1} 2i\right)} f_{q^* (l-2k)} & \text{ if } l>1.
        \end{cases}
    \end{equation}
Moreover, for $i,j \in \{0, \dots, r-1\}$, $q^*i \equiv q^*j \mod r $ if and only if $i = j$.
\end{lem}

\begin{proof}
    Since $(r,q)=1$, there is a unique $q^* \in \mathbb{Z}_r$ such that $qq^* \equiv 1 \mod r$.
    Using the symmetries of $f_l$ and substituting in $q^*l$, we have
    \begin{align}\label{recursive_e_i}
        e_{1+l} &= f_{q^*l} - A^{2pq^*l} e_{1-l} = f_{q^*l} + A^{2pq^*l} e_{l-1} = f_{q^*l} + A^{2pq^*l} e_{1 + (l-2)}.
    \end{align}
    We can then apply Equation (\ref{recursive_e_i}) iteratively to express the $e_i$'s in terms of
    the $f_j$'s.
    

\begin{align*}
    e_{l+1} =f_{q^* l} + A^{2pq^* l}f_{q^* (l-2)} + A^{2pq^* (2l-2)}f_{q^* (l-4)} + \cdots + A^{2pq^*\left(\lfloor l / 2 \rfloor l- \sum_{i=0}^{\lfloor l / 2 \rfloor-1} 2i\right)}f_{q^*(l-2\lfloor l / 2 \rfloor)}
\end{align*}    
    for $l \in \{2, \dots, m-1 \}$. When $l=0$, by definition, 
    $e_1=f_0.$
    When  $l=1$, 
    Equation (\ref{recursive_e_i}) yields that
    $e_2 = f_{q^*}.$
    For any $l \in \{0, \dots, m-1 \}$,
    the iterative use of Equation (\ref{recursive_e_i}) to express $e_{l+1}$ terminates when the
    final term is a scalar multiple of either $e_1$ or $e_2$, depending on the parity of $l$. 
    
    For the final statement, note that $(q^*, r)=1$. This means there is a group isomorphism between the cyclic groups $\{q^*k \mod r| k \in \mathbb{Z}_r\}$ and $\mathbb{Z}_r$ sending the indices $q^*k \mod r$ in Equation (\ref{e_as_f_eqn}) to distinct $j$ for $j \in \{0, 1, \dots, r-1 \}$. Since $q^*k \mod r$ are distinct for $k\in \mathbb{Z}_r$, this shows that $q^*i \equiv q^*j \mod r$ if and only if $i=j$. 
\end{proof}

In order to prove Proposition \ref{2NBasisSpans}, we will use the following lemma. The proof of this
lemma is given in Subsection \ref{subsec: TechLemmaPfs}. 

\begin{lem}\label{f_m_in_span} Suppose $r=2m+1 \geq 3$ is coprime to $q$. Then 
    \[
    \Tilde{f}_m = \sum_{j=0}^{m-1} C_j f_j,
    \]
    where $C_j \in \mathbb{C}$ such that $|C_j|= 1$ for $j \in \{0, \dots,  m-1\}$.
\end{lem}

\begin{proof}[Proof of Proposition \ref{2NBasisSpans}]
It suffices to show that $R_m$ is nonsingular, in which case $R_m$ corresponds to the basis
transformation $\Tilde{F}_m \rightarrow \{e_1,\dots,e_m\}$. To establish nonsingularity, we will
give a presentation of $R_m^{-1}$ by expressing $e_i$, for $i \in \{1, \dots, m\}$, in terms of $f_j$ where $j \in \{0,
\dots m-1\}$.

By Lemma \ref{e_i_as_f_ls}, each $e_i$, for $i \in \{1,\dots, m\}$, can be written in terms of $f_j$ where $j\in
\mathbb{Z}$. These $f_j$'s reduce to $f_l$'s, where $l \in \{0, \dots, m\}$, using the above
symmetries. This means that $Span\{e_1, \dots, e_m\}$ of dimension $m$ is contained in $Span\{f_0,
\dots , f_m\}$, a vector space of dimension at most $m+1$. 

Lemma \ref{f_m_in_span} implies that $\Tilde{f}_m \in Span\{f_0, \dots, f_{m-1}\}$, which means that
$$Span\{f_0, \dots, f_{m-1}\} = Span\{f_0, \dots, f_{m}\} \supseteq Span\{e_1, \dots, e_m\}.$$ 
Since $\{e_1, \dots, e_m\}$ is a basis for the $m$-dimensional vector space $RT_r(T^2)$, $\{f_0, \dots, f_{m-1}\}$ is a set of $m$ vectors, and $Span\{f_0, \dots, f_{m-1}\} \supseteq Span\{e_1, \dots, e_m\}$, then
$$Span\{f_0, \dots, f_{m-1}\} = Span\{e_1, \dots, e_m\}= RT_r(T^2).$$
From this, we conclude that $\{f_0, \dots, f_{m-1}\}$ is also a basis for $RT_r(T^2)$. Since $\{f_0, \dots, f_{m-1}\}$ is a basis, this implies  that $R_m^{-1}$ is a change-of-basis matrix and is nonsingular, therefore, $R_m$ is nonsingular. 

In order to bound the operator norm $|||R_m^{-1}|||$, we study the presentation of $R_m^{-1}$ more
closely. By Lemma \ref{e_i_as_f_ls}, we may express each $e_i$ as
$$e_i = \sum_{j=0}^{r-1} B_j^i f_j,$$
where $B_j^i$ is either zero or a root of unity and the summands correspond to the reduction of each index modulo $r$. We remark that since $B_j^i$ is either zero or a root of unity, $|B_j^i| \leq 1$. Now after applying the symmetry $f_l = A^{2pl} f_{-l} = A^{2pl} f_{r-l}$ for any $l>m$, we may express $e_i$ as
$$e_i = \sum_{j=0}^{m} (B_j^i + A^{2p(r-j)}B_{r-j}^i)f_j = \sum_{j=0}^{m} D_j^i f_j,$$
where $D_j^i = B_j^i + A^{2p(r-j)}B_{r-j}^i$ and $|D_j^i| \leq 2$. Additionally, by Lemma \ref{f_m_in_span}, we know that the coefficient of any summand of $f_m$ in terms of the basis $\{f_0, \dots, f_{m-1} \}$ is $C_j^i$ with $|C_j^i| =1$. This means we may write 

\begin{align*}
e_i &= \left(\sum_{j=0}^{m-1} D_j^i f_j\right) + D_m^i f_m \nonumber\\
&= \left(\sum_{j=0}^{m-1} D_j^i f_j\right) + D_m^i \left(\sum_{k = 0}^{m-1} C_k^i f_k\right) \nonumber\\
&= \sum_{j=0}^{m-1} \left[D_m^i C_j^i + D_j^i \right] f_j. \nonumber\\
& = \sum_{j=0}^{m-1} E_j^i f_j
\end{align*}
where $E_j^i = D_m^i C_j^i + D_j^i.$
Note that since $|D_j^i| \leq 2$ and $|C_j^i|=1$, we have
$$ | E_j^i| = |D_m^i C_j^i + D_j^i| \leq |D_m^i C_j^i| + |D_j^i| = |D_m^i||C_j^i| + |D_j^i| \leq 4.$$
Hence every entry of $R_m^{-1}$ has modulus bounded above by $4$. For any complex unit vector $v = [v_0, \dots, v_{m-1}]^T$ such that $|v_i| \leq 1$ for $i \in \{0, \dots m-1\}$, the Cauchy-Schwartz inequality implies that 
\begin{align*}
\|R_m^{-1} v \| & = \left\| \left[\sum_{i=0}^{m-1}E_0^iv_{i}, \dots, \sum_{i=0}^{m-1}E_{m-1}^iv_i \right]^T \right\| \\
& = \left(\left|\sum_{i=0}^{m-1}E_0^iv_{i}\right|^2 + \dots + \left|\sum_{i=0}^{m-1}E_{m-1}^iv_i\right|^2 \right)^{\frac{1}{2}} \\
& \leq \left(\sum_{i,j=0}^{m-1}\left|E_j^i\right|^2 \left|v_i\right|^2 \right)^{\frac{1}{2}} 
\leq \left(\sum_{i,j=0}^{m-1}\left|E_j^i\right|^2  \right)^{\frac{1}{2}} \\
& \leq \left(\sum_{i,j=0}^{m-1} 4^2  \right)^{\frac{1}{2}} =
 \left(\sum_{i,j=0}^{m-1} 16 \right)^{\frac{1}{2}} = \left(16m^2\right)^\frac{1}{2} = 4m.
\end{align*}
This shows that 
\[
||R_m^{-1} v|| \leq O(m),
\]
so the operator norm $|||R_m^{-1}|||$ is bounded polynomially.

Lastly, by the Cabling Formula in Theorem \ref{CablingFormula} and the definition of $f_l$, the coefficient of $f_l$ in $RT_r(C_{p,q})(e_i)$ is given by
\begin{align*}
RT_r(C_{p,q})(e_i) =& A^{\frac{qp}{2}(i^2-1)} \sum_{l \in T_i} A^{-p\left(\frac{q}{2}l^2 + l\right)} \Tilde{f}_l,
\end{align*}
where $T_i = \{0,2,\dots, i-1\}$ for odd $i$ and $T_i = \{1,3,\dots, i-1\}$ for even $i$.
\end{proof}

In order to prove Proposition \ref{RNotInvert}, we establish the following definitions.

For $1 \leq l \leq m$, define $f_l^{\pm} := e_{ql \pm 1}$. Observe that $f_l = f_l^+ - A^{2pl}
f_l^-$ for $1\leq l \leq m$. In addition, define $\Tilde{f}_l^{\pm}$ to be the quotient of
$f_l^{\pm}$ under the symmetries $e_{-i} = -e_i$ for any $i \geq 0$ and $e_{i+kr} = (-1)^ke_i$ for
any $k \in \mathbb{Z}$. We will use the convention that $f_0^+ = f_0 = e_1$ and $f_0^- = 0$.

Recall that for each $l$, $ql\pm 1 = kr+j$ for some non-negative integers $k, j$ where $0 \leq j
<r$. This means that up to sign, 
\begin{align}\label{fpm}
\Tilde{f}_l^{\pm} = 
\begin{cases}
    e_j = e_{ql-kr\pm 1} & 0 \leq j\leq m \\
    e_{r-j} = e_{(k+1)r-ql\mp 1} & m+1 \leq j< r.
\end{cases} 
\end{align}

We can now prove Proposition \ref{RNotInvert}.

\begin{proof}[Proof of Proposition \ref{RNotInvert}]
We will repeat the argument given in the proof of Theorem \ref{RTinvertible}. Suppose $(r,q) = d>1$,
then there are coprime $q'$ and $r'$ such that $q=dq'$ and $r=dr'$. We claim that $row(nd)$ of $R_m$
consists of only zeros for each $n$. Suppose some $\Tilde{f}^{\pm}_l = e_j = e_{nd}$, then $ql-kr\pm
1 = nd$. This implies that $d(q'l-r'k-n) = \mp 1$  which is a contradiction. Similarly, if
$\Tilde{f}^{\pm}_l = e_{r-j} = e_{nd}$, then $d((1+k)r'-q'l-n) = \pm 1$ which is also a
contradiction. This means that $row(nd) = [0,\dots, 0]$, thus $R_m$ is singular. 
\end{proof}

\subsection{Proof of Lemma \ref{f_m_in_span}}\label{subsec: TechLemmaPfs} In this subsection, we provide a
proof for Lemma \ref{f_m_in_span}. We use the notation introduced in Subsection \ref{subsec:
suppProp}.

\begin{remark}
For the following arguments, we use the convention that equalities between vectors $e_i$ are
necessarily taken up to sign. This ultimately has no effect on the arguments for Proposition
\ref{2NBasisSpans} and Theorem \ref{RTinvertible}.
\end{remark}


Recall $R_m$ is the $(m \times m)$-matrix with columns corresponding to $\Tilde{F}_m = \{f_0, \dots,
f_{m-1} \}$. We also define $S_m$ to be the $(m \times (m+1))$-matrix obtained by appending the
column corresponding to $\Tilde{f}_m$ to $R_m$. The following technical lemmas will be used in the
proof of Lemma \ref{f_m_in_span}.

\begin{lem}\label{matrixlem} Suppose $r=2m+1 \geq 3$ is coprime to $q$, and let $q^*$ be the multiplicative inverse
    of $q$ in the ring $\mathbb{Z}_r$. Then
    \begin{enumerate}[(i)]
        \item Each column of $S_m$ has at most two nonzero entries. Moreover, for each column with
        two nonzero entries, their corresponding row indices differ by at most $2$.
        \label{matrixlem_i} 
        \item Let 
        \begin{equation*}
        l^* :=
            \begin{cases}
                q^* & \text{if } q^* \leq m\\
                r - q^* & \text{if } q^* > m.
            \end{cases}
        \end{equation*}
         Then in $S_m$, $col(1) = [1,0,\dots,0]^T$ and $col(l^*+1) = [0,D_{l^*},0,\dots,0]^T$ where
         $D_{l^*}$ is a root of unity. Moreover, every other column of $S_m$ has exactly two nonzero
         entries which are roots of unity. \label{matrixlem_ii}
    \end{enumerate}
\end{lem}

\begin{proof}
    \textbf{Part $(\ref{matrixlem_i})$:} Each column of $S_m$ corresponds to the reduced vector
    $\Tilde{f}_l$, $0\leq l \leq m$. Since $f_l$ is a linear combination of at most two vectors in
    $Span\{e_1, e_{ql+1}, e_{ql-1}\}_{l=1}^{m-1}$, there are at most two nonzero entries in
    $col(l+1)$.

Now suppose the index of $f_l^+$ is $ql+1 = kr+j$, where $0 \leq j <r$. Then the index of $f_l^-$ is
$ql-1 = kr+j-2 = k'r+j'$, where either $(k',j')=(k-1,r+j-2)$ (for $j\in \{0,1\}$) or
$(k',j')=(k,j-2)$ (for $j\geq 2$). We split into cases:
\begin{itemize}
    \item If $j=0$, then $j' = r-2$, $\Tilde{f}_l^+ = e_0 = 0$, and $\Tilde{f}_l^- = e_2$. 
    \item If $j=1$, then $j'=r-1$, $\Tilde{f}_l^+ = e_{ql-kr + 1} = e_1$, and $\Tilde{f}_l^- =
    e_{(k+1)r-ql +1} = e_{1}$. This implies that $l = \frac{kr}{q}$. Since $l \in \mathbb{Z}$ and
    $r,q $ are  coprime, $k = qn$, for some $n\geq 0$. However, if $n\geq 1$, we have $l \geq
    1+r>m$, which is a contradiction. Thus $n=0$, so $l=0$, corresponding to $col(1)$.
    \item If $2\leq j \leq m$, then $\Tilde{f}_l^+ = e_j \neq e_{j-2} = \Tilde{f}_l^-$.   
    \item If $j=m+1$, then $j' = m-1$ and $\Tilde{f}_l^+ = e_{m} \neq e_{m-1} = \Tilde{f}_l^-$.
    \item If $j=m+2$, then $j' = m$ and $\Tilde{f}_l^+ = e_{m-2} \neq e_{m} = \Tilde{f}_l^-$.
    \item If $m+3 \leq j <r$, then $\Tilde{f}_l^+ = e_{r-j} \neq e_{r-j+2} = \Tilde{f}_l^-$.
\end{itemize}
This implies that the row indices of the nonzero entries in each column differ by at most two for
every column except $col(1)$. In particular, the only case where the row indices differ by exactly
$1$ occurs when $j = m+1$.

\textbf{Part $(\ref{matrixlem_ii})$:} Since $e_{-i} = -e_i$, we have $f_l = e_{1+ql} + A^{2pl}
e_{1-ql}$ for $1 \leq l \leq m$. Note that $col(l+1)$ has exactly one nonzero entry if and only if
one of the following occurs:
\begin{enumerate}[(1)]
    \item Either $1 + ql$ and $1 - ql$ are   equal or opposite modulo $r$.
    \item Either $1 + ql$ or $1 - ql$ vanishes modulo $r$.
\end{enumerate}

Case $(1)$ occurs if and only if $l=0$, corresponding to $\Tilde{f}_0 = f_0 = e_0$. In this case,
$col(1) = [1, 0, \dots, 0]^T$. 

Case $(2)$ occurs if and only if either $l = q^*$ or $l = -q^*$ modulo $r$. Define     
\begin{equation*}
    l^* :=
        \begin{cases}
            q^* & \text{if } q^* \leq m\\
            r - q^* & \text{if } q^* > m.
        \end{cases}
    \end{equation*}
Note that if $ql \pm 1$ vanishes, then $|ql \mp 1| = 2$. Define $D_{l^*}$ to be the coefficient of
the vector $e_{|ql \mp 1|} = e_2$ obtained from Equation (\ref{Coeffofl}). This means that $col(l^*+1)$ is
the unique column with exactly one nonzero entry except for $col(1)$. 

Finally, the conclusion follows from the uniqueness of $l^*$ and Part $(\ref{matrixlem_i})$.
\end{proof}

The second technical lemma makes use of Lemma \ref{matrixlem} in its proof.

\begin{lem}\label{matrixlem2} Suppose $r = 2m+1 \geq 3$ is coprime to $q$. Then 
    \begin{enumerate}[(i)]
    \item Each row of $S_m$ has exactly two nonzero entries. \label{matrixlem2_i}
    
    \item There is a unique $l'$, $1 \leq l' \leq m$, such that $col(l'+1) =
    [0,\dots,0,D_{l'},E_{l'}]^T$, where $D_{l'},E_{l'}$ are roots of unity.
    \label{matrixlem2_ii} 
    \end{enumerate}
\end{lem}

The following lemma will be useful in the proof of Lemma \ref{matrixlem2}.

\begin{lem}\label{indlem} Suppose $r \geq 3$ is coprime to $q$, and let $g_l^{\pm}:= ql-kr\pm 1$ and
$h_l^{\pm} = (1+k)r-ql \mp 1$. Then for $0 \leq l_1, l_2 \leq m$ with $l_1 \neq l_2$, 
\begin{enumerate}[(i)]
    \item $g_{l_1}^{\pm} = g_{l_2}^{\pm}$, $g_{l_1}^{\pm} = h_{l_2}^{\mp}$, and $h_{l_1}^{\pm} =
    h_{l_2}^{\pm}$ do not have integer solutions, \label{indlem(a)}
    \item $g_{l_1}^{\pm} = g_{l_2}^{\mp}$, $g_{l_1}^{\pm} = h_{l_2}^{\pm}$, and $h_{l_1}^{\pm} =
    h_{l_2}^{\mp}$ may each have integer solutions. \label{indlem(b)}
\end{enumerate}
\end{lem}

\begin{proof}
    Note $g_l^{\pm}$ and $h_l^{\pm}$ encode the two families of indices of the reduced vectors
    $\Tilde{f}_l^{\pm}$ given in Equation (\ref{fpm}). There are six equations relating pairs of
    expressions in $\{g_l^+, g_l^-, h_l^+, h_l^-\}$.
   
    \textbf{Part} $(\ref{indlem(a)})$: This follows from the fact that $(r,q)=1$ and the
    bounds on $l_1$ and $l_2$. We show the case $g_{l_1}^{\pm} = g_{l_2}^{\pm}$ and note that the other two cases follow analogously. Assume for distinct $l_1, l_2 \in \{0, \dots, m\}$ and $k_1, k_2\in \mathbb{Z}$ that $ql_1 - k_1 r \pm 1 = ql_2 - k_2 r \pm 1$. This implies 
    \[
    k_1-k_2 = \frac{q(l_1-l_2)}{r} \in \mathbb{Z}.
    \]
    Since $l_1 \neq l_2$ and $(r,q)=1$, $l_1-l_2$ must have a nontrivial factor of $r$, which contradicts the bounds on $l_1$ and $l_2$.
   
    \textbf{Part} $(\ref{indlem(b)})$: We have the following:
    \begin{itemize}
       \item $g_{l_1}^{\pm} = g_{l_2}^{\mp}$ if and only if $q(l_1-l_2) = (k_1-k_2)r \mp 2$,
       \item $g_{l_1}^{\pm} = h_{l_2}^{\pm}$ if and only if $q(l_1+l_2) = (1+k_1+k_2)r \mp 2$, and
       \item $h_{l_1}^{\pm} = h_{l_2}^{\mp}$ if and only if $q(l_1-l_2) = (k_1-k_2)r \mp 2$.
   \end{itemize}
   All three of these equations may have integer solutions for $l_1,l_2 \in \{0, \dots, m\}$.
\end{proof}

\begin{proof}[Proof of Lemma \ref{matrixlem2}] 
    We will use the same notation as in the proof of Lemma \ref{matrixlem} and in Lemma
    \ref{indlem}. 
    
    \textbf{Part $(\ref{matrixlem2_i})$:} It is a corollary of Lemma \ref{indlem} that every row of
    $S_m$ has at most two nonzero entries. In particular, let $(l_1, l_2)$ be an integral solution
    to one of the equations of Lemma \ref{indlem} Part $(\ref{indlem(b)})$. Suppose $l_3 \in \{0,
    \dots, m-1\}$ is such that $(l_1, l_3)$ and $(l_2, l_3)$ are both solutions to equations in
    Lemma \ref{indlem} Part $(\ref{indlem(b)})$. Then by Lemma \ref{indlem} Part
    $(\ref{indlem(a)})$, either $l_3 = l_1$ or $l_3 = l_2$.

    Note that by Lemma \ref{matrixlem} Part $(\ref{matrixlem_ii})$, $S_m$ has exactly $2m$ nonzero entries since there are 2 in each
    column other than $col(1)$ and $col(l^*+1)$, which each have exactly 1. This means that every
    row of $S_m$ must have exactly 2 nonzero entries.

    \textbf{Part $(\ref{matrixlem2_ii})$:} In the proof of Lemma \ref{matrixlem} Part
    $(\ref{matrixlem_i})$, we saw that the only value of $j$ corresponding to a column with the
    nonzero row entry indices differing by 1 is $j=m+1$. By Part $(\ref{matrixlem2_i})$,  $row(m)$ of
    $S_m$ has exactly 2 nonzero entries. This implies that there are some $l_1,l_2$ such that
    $col(l_1+1)$ has nonzero entries in $row(m)$ and $row(m-1)$ and $col(l_2+1)$ has nonzero entries
    in $row(m)$ and $row(m-2)$. Take $l' = l_1$. Finally, define $D_{l'}$ and $E_{l'}$ to be the
    coefficients of the vectors $e_{m-1}$ and $e_m$ defined by Equation (\ref{Coeffofl}),
    respectively. Note that if $m=2$, $l_2 = l^*$ and $col(l_2+1)$ has only 1 nonzero entry.
\end{proof}

Lastly, we are ready to prove Lemma \ref{f_m_in_span}.

\begin{proof}[Proof of Lemma \ref{f_m_in_span}]
The last column $col(m+1)$ of the matrix $S_m$ represents the reduced vector $f_m$ written in terms of the basis $\{e_1,\dots, e_m\}$. We will prove Lemma \ref{f_m_in_span} by showing that $col(m+1)$ can be written as a linear combination of the first $m$ columns. From this linear combination, we will see that the coefficients will have the required bounds from the statement. 



We claim that $col(m+1)$ of $S_m$ can be written as a linear combination of elements in $\{f_0, \dots, f_{m-1}\}$. From Lemma \ref{matrixlem} Part $(\ref{matrixlem_ii})$, if $l^*=m$, then $col(m+1)$ has exactly one nonzero entry, and if $l^* <m$, then $col(m+1)$ has exactly two nonzero entries.

\noindent \textbf{Case 1:}

We first consider the case $l^*=m$. Here, the nonzero entry of $col(m+1)$ lies in $row(2)$. This implies that $\Tilde{f}_m \in Span\{e_1, \dots, e_m\}$ has a scalar of $e_{2}$ as a summand. By Lemma \ref{matrixlem2} Part $(\ref{matrixlem2_i})$, we know that there is exactly one other nonzero entry in $row(2)$ in some column $j_1$. From the argument of Lemma \ref{matrixlem} Part $(\ref{matrixlem_i})$, there exists a nonzero entry in $row(4)$ of $col(j_1)$. Lemma \ref{matrixlem2} Part $(\ref{matrixlem2_i})$ implies there exists a nonzero entry in some column $j_2$ and $row(4)$. From the argument of Lemma \ref{matrixlem} Part $(\ref{matrixlem_i})$, there exists a nonzero entry in $row(6)$ of $col(j_2)$. Again, we pick the other nonzero entry of $row(6)$ which lies in some column $j_3$. Note that $col(j_3)$ cannot be equal to any of the previous columns. If it were a previous column, it would contradict our bound on the number of nonzero entries in a column. We continue this iteration until we reach either $row(m-1)$ or $row(m)$, depending on the parity of $m$. 

If $m-1$ is even, by Lemma \ref{matrixlem2} Part $(\ref{matrixlem2_ii})$, the next corresponding row with a nonzero entry will be $row(m)$ where $m$ is odd. Similarly, if $m$ is even, by Lemma \ref{matrixlem2} Part $(\ref{matrixlem2_ii})$, the next corresponding row with a nonzero entry will be $row(m-1)$ where $m-1$ is odd. Now when we continue the algorithm, our subsequent row indices will be odd and decrease by $2$ until we reach $row(1)$. By Lemma \ref{matrixlem2} Part $(\ref{matrixlem2_i})$ and Lemma \ref{matrixlem} Part $(\ref{matrixlem_ii})$, there exists a nonzero entry in $row(1)$ of $col(1)$, and it is the only nonzero entry in $col(1)$. Since every entry of our matrix is a root of unity by Lemma \ref{matrixlem} Part $(\ref{matrixlem_ii})$ and terminates at $row(1)$, scalars by roots of unity of the columns appearing in our sequence gives $f_m$  as a linear combination  of  elements of $\{f_0, \dots, f_{m-1}\}$ where all coefficients are roots of unity.  

\noindent \textbf{Case 2:}

Now suppose $col(m+1)$ has exactly two nonzero entries. We denote the row indices of these entries by $i_1^-$ and $i_1^+$, where $i_1^- < i_1^+$. By Lemma \ref{matrixlem2} Part $(\ref{matrixlem2_i})$, $row(i_1^-)$ has another nonzero entry in some other column $j_1^-$. Similarly, $row(i_1^+)$ has another nonzero entry in some column $j_1^+$. We make the following claim, which we prove at the end.

\noindent \textbf{Claim:} $j_1^- \neq j_1^+$. 

We will proceed similarly to the first case. Consider the column $col(j_1^+)$, which has exactly two nonzero entries and cannot correspond to either $col(1)$ or $col(l^*+1)$ since $i_1^+>2$. 
By Lemma \ref{matrixlem2} Part $(\ref{matrixlem2_i})$ and the claim, there exists another nonzero entry in some $row(i_2^+)$ of $col(j_1^+)$ such that $(i_2^+-i_1^+)\in \{-1,1,2\}$. The case when $(i_2^+-i_1^+) = -1$ corresponds to $i_1^+ = m$, and the case when $(i_2^+-i_1^+)=1$ corresponds to $i_1^+=m-1$. We now implement the same argument as the case with one entry in $col(m+1)$. Note that, in this procedure, we do not utilize any rows with index less than  $i_1^-$ with the same parity as $i_1^-$. If $i_1^-=m-1$ and $i_1^+=m$, they will  have different parities. In the other case, $i_1^-$ will have the parity of $i_k^+$ until we have a $k$ such that $(i_k^+-i_{k-1}^+) \in \{-1,1\}$. This implies that for all $k' \geq k$, $i_{k'}^+$ will have opposite parity to $i_1^-$. 

We now follow the same algorithm beginning with $row(i_1^-)$. By the claim, the indices of our subsequent rows $i_k^-$ must be decreasing. Otherwise, this would contradict Lemma \ref{matrixlem2} Part $(\ref{matrixlem2_i})$. 

Since both cases in total utilize every row exactly once, $f_m$ is given by a linear combination of elements of $\{f_0, \dots, f_{m-1}\}$ where, by Lemma \ref{matrixlem} Part $(\ref{matrixlem_ii})$, all coefficients are roots of unity.

\noindent \textbf{Proof of Claim:}

It now suffices to prove that $j_1^+\neq j_1^-$. By contradiction, let us assume that $j_1^+=j_1^-$, and we will denote $i_1=i_1^+$. 

If $i_1=m$, then either $i_1^-=m-1$ or $i_1^-=m-2$. If $i_1^-=m-1$, then since $j_1^+=j_1^-$, we will have two columns with nonzero entries in the last two rows. This contradicts Lemma \ref{matrixlem2} Part $(\ref{matrixlem2_ii})$, which states that there is a unique such column. If $i_1^-=m-2$, then there are two distinct columns with nonzero entries in $row(m-2)$ and $row(m)$. By Lemma \ref{matrixlem2} Part $(\ref{matrixlem2_ii})$, there must exist  a different column with nonzero entries in $row(m-1)$ and $row(m)$, which contradicts there being at most $2$ entries in $row(m)$. 

If $i_1=m-1$, then $i_1^-<m-1$, and there are no other columns with nonzero entries in $row(m-1)$ besides $col(j_1^+)$ and $col(m+1)$. By Lemma \ref{matrixlem2} Part $(\ref{matrixlem2_ii})$, there must exist  a different column with nonzero entries in $row(m-1)$ and $row(m)$, which contradicts there being at most $2$ entries in $row(m-1)$.

In the general case, we assume $i_1\leq m-2$, and we will define $i_2=i_1+2$.  Since $j_1^+=j_1^-$, $col(j_1^+)$ and $col(m+1)$ already have two nonzero entries. Since $i_2>2$, these entries cannot be in either $col(1)$ or $col(l^*+1)$ since they only have entries in the first two rows. This implies that the columns which correspond to nonzero entries in $row(i_2)$ must have exactly two nonzero entries in some columns $col(j_2^+)$ and $col(j_2^-)$ such that $j_2^+,j_2^- \notin \{j_1^+,m+1\}$. Since $row(i_1)$ has two nonzero entries in $col(j_1^+)$ and $col(m+1)$, the other nonzero entries in $col(j_2^+)$ and $col(j_2^-)$ must be in some $row(i_3)$, where $i_3 - i_2 \in \{-1, 1,2\}$.
\begin{itemize}
    \item If $i_3-i_2=-1$, we have $i_2=m$ and $i_3 = m-1$. Here, we reach the same contradiction as when $i_1=m$ and $i_1^-=m-1$.
    \item If $i_3-i_2 = 1$, then $i_2=m-1$ and $i_3 = m$. This gives the same contradiction as when $i_1=m$ and $i_1^-=m-1$. 
    \item If $i_3-i_2=2$ with $i_2=m-2$ and $i_3 = m$, then our argument is the same as when $i_1=m$ and $i_1^-=m-2$.
\end{itemize}
 Finally, we consider when $i_3-i_2=2$ and $i_3 \neq m$. In this case, we can continue to iterate the same algorithm until we reach the same contradictions.




\end{proof}

\section{Further Directions}\label{Sec: FD} The primary approach of this paper utilizes the
invertibility of the operator $RT_r$ on the cable space $C_{p,q}$ as well as a polynomial bound on
its operator norm. The same methodology could apply in the context for the operator $RT_r$ for other
cable spaces.

Although the technique may apply in the case when the cable space has positive simplicial volume, a
more natural approach would be to generalize our argument to other cable spaces with simplicial
volume zero. For example, we may consider the manifold defined as follows. Let $N=\Sigma_{g,2}
\times S^1$ where $\Sigma_{g,2}$ is a orientable compact genus $g$ surface with $2$ boundary
components. Now let $\{x_i\}_{i=1}^m \subset \Sigma_{g,2}$ such that $\{ x_i\}_{i=1}^m \times S^1$
is a collection of $m$ vertical fibers in $N$. We define the Seifert cable space $C\left( s_1,\dots
s_m\right)$ where $s_i = \frac{p_i}{q_i} \in \mathbb{Q}$ to be the manifold obtained by performing
$s_i$-Dehn surgery along the $i$-th vertical fiber in $N$.

If an analogous result to Theorem \ref{RTinvertible} holds for the Seifert cable space $C\left(
s_1,\dots s_m\right)$, the corresponding Theorem \ref{mainthm} will also follow as well as its
applications to Conjecture \ref{GenCYvol}. Similar to the constraint of Theorem \ref{RTinvertible}
where $r$ and $q$ must be coprime, the analogous result for the Seifert cable space may require a
related caveat. This leads to the following concluding question.

\begin{ques}
Is $RT_r(C\left( s_1,\dots s_m\right))$ invertible when $r$ is sufficiently large and coprime to
every $q_i$?
\end{ques}

\bibliographystyle{abbrv}

\bibliography{biblio}

\begin{thebibliography}{10}

\bibitem{Belletti2021}
G.~Belletti.
\newblock The maximum volume of hyperbolic polyhedra.
\newblock {\em Trans. Amer. Math. Soc.}, 374(2):1125--1153, 2021.

\bibitem{growth6j}
G.~Belletti, R.~Detcherry, E.~Kalfagianni, and T.~Yang.
\newblock Growth of quantum {$6j$}-symbols and applications to the volume
  conjecture.
\newblock {\em J. Differential Geom.}, 120(2):199--229, 2022.

\bibitem{TVCompact}
R.~Benedetti and C.~Petronio.
\newblock On {R}oberts' proof of the {T}uraev-{W}alker theorem.
\newblock {\em J. Knot Theory Ramifications}, 5(4):427--439, 1996.

\bibitem{BHMVKauffman}
C.~Blanchet, N.~Habegger, G.~Masbaum, and P.~Vogel.
\newblock Topological quantum field theories derived from the {K}auffman
  bracket.
\newblock {\em Topology}, 34(4):883--927, 1995.

\bibitem{chen2018volume}
Q.~Chen and T.~Yang.
\newblock Volume conjectures for the {R}eshetikhin-{T}uraev and the
  {T}uraev-{V}iro invariants.
\newblock {\em Quantum Topol.}, 9(3):419--460, 2018.

\bibitem{DetcherryCabling}
R.~Detcherry.
\newblock Growth of {T}uraev-{V}iro invariants and cabling.
\newblock {\em J. Knot Theory Ramifications}, 28(14):1950041, 8, 2019.

\bibitem{detcherry2019gromov}
R.~Detcherry and E.~Kalfagianni.
\newblock Gromov norm and {T}uraev-{V}iro invariants of 3-manifolds.
\newblock {\em Ann. Sci. \'{E}c. Norm. Sup\'{e}r. (4)}, 53(6):1363--1391, 2020.

\bibitem{colJvolDKY}
R.~Detcherry, E.~Kalfagianni, and T.~Yang.
\newblock Turaev-{V}iro invariants, colored {J}ones polynomials, and volume.
\newblock {\em Quantum Topol.}, 9(4):775--813, 2018.

\bibitem{Gromov}
M.~Gromov.
\newblock Volume and bounded cohomology.
\newblock {\em Inst. Hautes \'{E}tudes Sci. Publ. Math.}, 56:5--99 (1983),
  1982.

\bibitem{JacoShalen}
W.~H. Jaco and P.~B. Shalen.
\newblock Seifert fibered spaces in {$3$}-manifolds.
\newblock {\em Mem. Amer. Math. Soc.}, 21(220):viii+192, 1979.

\bibitem{Johannson}
K.~Johannson.
\newblock {\em Homotopy equivalences of {$3$}-manifolds with boundaries},
  volume 761 of {\em Lecture Notes in Mathematics}.
\newblock Springer, Berlin, 1979.

\bibitem{KashaevVolConj}
R.~M. Kashaev.
\newblock The hyperbolic volume of knots from the quantum dilogarithm.
\newblock {\em Lett. Math. Phys.}, 39(3):269--275, 1997.

\bibitem{Kumar21}
S.~Kumar.
\newblock Fundamental shadow links realized as links in ${S}^3$.
\newblock {\em Algebraic \& Geometric Topology}, 21(6):3153--3198, 2021.

\bibitem{kumarMelby21}
S.~Kumar and J.~M. Melby.
\newblock Asymptotic additivity of the {T}uraev-{V}iro invariants for a family
  of $3$-manifolds.
\newblock {\em {J}ournal of the {L}ondon {M}athematical {S}ociety}, to appear,
  2022.

\bibitem{MortonCabling}
H.~R. Morton.
\newblock The coloured {J}ones function and {A}lexander polynomial for torus
  knots.
\newblock {\em Math. Proc. Cambridge Philos. Soc.}, 117(1):129--135, 1995.

\bibitem{MurakamiMurakamiVolConj}
H.~Murakami and J.~Murakami.
\newblock The colored {J}ones polynomials and the simplicial volume of a knot.
\newblock {\em Acta Math.}, 186(1):85--104, 2001.

\bibitem{PerelmanEntropy}
G.~Perelman.
\newblock The entropy formula for the {R}icci flow and its geometric
  applications.
\newblock {\em arXiv preprint at arXiv:math/0211159}, 2002.

\bibitem{PerelmanFinite}
G.~Perelman.
\newblock Finite extinction time for the solutions to the {R}icci flow on
  certain three-manifolds.
\newblock {\em arXiv preprint at arXiv:math/0307245}, 2003.

\bibitem{PerelmanRicci}
G.~Perelman.
\newblock Ricci flow with surgery on three-manifolds.
\newblock {\em arXiv preprint at arXiv:math/0303109}, 2003.

\bibitem{ReshetikhinTuraev1991}
N.~Reshetikhin and V.~G. Turaev.
\newblock Invariants of {$3$}-manifolds via link polynomials and quantum
  groups.
\newblock {\em Invent. Math.}, 103(3):547--597, 1991.

\bibitem{RobertsSkein}
J.~Roberts.
\newblock Skein theory and {T}uraev-{V}iro invariants.
\newblock {\em Topology}, 34(4):771--787, 1995.

\bibitem{ThurstonGT3manifolds}
W.~P. Thurston.
\newblock {\em The geometry and topology of three-manifolds}.
\newblock Princeton University Math Department Notes, 1979.

\bibitem{ThurstonGeomConj}
W.~P. Thurston.
\newblock Three-dimensional manifolds, {K}leinian groups and hyperbolic
  geometry.
\newblock {\em Bull. Amer. Math. Soc. (N.S.)}, 6(3):357--381, 1982.

\bibitem{TuraevViro}
V.~G. Turaev and O.~Y. Viro.
\newblock State sum invariants of {$3$}-manifolds and quantum {$6j$}-symbols.
\newblock {\em Topology}, 31(4):865--902, 1992.

\bibitem{WongWhitehead}
K.~H. Wong.
\newblock Asymptotics of some quantum invariants of the {W}hitehead chains.
\newblock {\em arXiv preprint at arXiv:1912.10638}, 2019.

\bibitem{wong2020WHFig8}
K.~H. Wong.
\newblock Volume conjecture, geometric decomposition and deformation of
  hyperbolic structures.
\newblock {\em arXiv preprint at arXiv:1912.11779}, 2020.

\end{thebibliography}

\Addresses

\end{document}